\documentclass{article}
\usepackage{color}
\usepackage{amsfonts}
\usepackage{amssymb}
\usepackage{xcolor}
\usepackage{amsthm}
\usepackage{amsmath}
\usepackage{graphicx}
\DeclareFontFamily{U}{mathx}{\hyphenchar\font45}
\DeclareFontShape{U}{mathx}{m}{n}{
      <5> <6> <7> <8> <9> <10>
      <10.95> <12> <14.4> <17.28> <20.74> <24.88>
      mathx10
      }{}
\DeclareSymbolFont{mathx}{U}{mathx}{m}{n}
\DeclareFontSubstitution{U}{mathx}{m}{n}
\DeclareMathAccent{\widecheck}{0}{mathx}{"71}
\DeclareMathAccent{\wideparen}{0}{mathx}{"75}
\providecommand{\U}[1]{\protect \rule{.1in}{.1in}}
\definecolor{ashgrey}{rgb}{0.7, 0.75, 0.81}
\newtheorem{theorem}{Theorem}

\newtheorem{definition}[theorem]{Definition}

\newtheorem{lemma}[theorem]{Lemma}
\newtheorem{notation}[theorem]{Notation}

\newtheorem{proposition}[theorem]{Proposition}

\begin{document}
     \title{Geometries on Polygons in the unit disc}
\author{Charalampos Charitos, Ioannis Papadoperakis \and and Georgios
Tsapogas \\
Agricultural University of Athens}
\maketitle
     \begin{abstract}
     For a family $\mathcal{C}$ of properly embedded curves in the 2-dimensional
disk $\mathbb{D}^{2}$ satisfying certain uniqueness properties, we consider
convex polygons $P\subset \mathbb{D}^{2}$ and define a metric $d$ on $P$
such that $(P,d)$ is a geodesically complete metric space whose geodesics
are precisely the curves $\left\{ c\cap P\bigm\vert c\in \mathcal{C}\right\}.$

Moreover, in the special case $\mathcal{C}$ consists of all Euclidean lines,
it is shown that $P$ with this new metric is not isometric to any convex
domain in $\mathbb{R} ^{2}$ equipped with its Hilbert metric.

We generalize this construction to certain classes of  uniquely geodesic metric spaces
homeomorphic to $\mathbb{R}^{2}.$
\textit{Subjclass:} 51F99
     \end{abstract}
     \maketitle
\section{Introduction}

Hilbert's $4^{\mathrm{th}}$ problem was asking for a characterization of all metrics on a convex subset  of Euclidean space for which straight lines are geodesics.
Before Hilbert, Beltrami in \cite{Beltrami} had already shown that the
unit disc in the plane, with the Euclidean chords taken as geodesics of
infinite length, is a model of the hyperbolic geometry. However, Beltrami did
not give a formula for this distance, and this led Klein in \cite{Klein}
to express the distance in the unit disc in terms of the cross radio. Over the years,
Hilbert's fourth problem became a very active research
area and Hilbert's metric defined on convex domains using cross ratio played an important role.
It was gradually realized that the discovery of all metrics satisfying Hilbert's problem was not plausible.
Consequently, each metric resolving Hilbert's problem defines a new
geometry worth to be studied. A very important class of such metrics,
defined by means of the cross ratio, are referred to as \textit{Hilbert
metrics} and play a central role in this research area, see \cite{Tro} for the origin of Hilbert geometry.

Among the prominent mathematicians worked on the Hilbert's fourth problem, it
is worthy to mention Busemann and Pogorelov, see for instance \cite{Buseman},
\cite{Pogorelov1}, \cite{Pogorelov2}. The ideas of the latter to solve
Hilbert's fourth problem came from Busemann, who introduced
integral geometry techniques to approach Hilbert's problem.

Hilbert's $4^{\mathrm{th}}$ problem admits various formulations as well as generalizations.
One of them is to find metrics on subsets of the plane with prescribed geodesics. Blaschke and  Bol in \cite{BB} were the first to consider such problems while Busemann and Salzman \cite{BS} focused on the whole Euclidean plane.

Convex polytopes constitute an important class of convex domains whose Hilbert geometry has been extensively studied, see for example \cite{CVV}, \cite{Ver}.  In this work we focus on convex polygons in the unit disk $\mathbb{D}^{2}$ in $\mathbb{R}^{2}.$  We consider a class $\mathcal{C}$ of continuous curves in $\mathbb{D}^{2}\cup\partial \mathbb{D}^{2}$ satisfying natural assumptions (see properties (C1)-(C3) below) analogous to those satisfied by the class of geodesics in a uniquely geodesic metric space. Then for any convex polygon $P\subset \mathbb{D}^{2}$ we  explicitly construct a metric on $P,$ in fact a family of metrics, such that $(P,d)$ is a geodesically complete metric space whose geodesics are precisely the curves $\left\{ c\cap P\bigm\vert c\in \mathcal{C}\right\}.$ The construction uses a family of pseudo-metrics on  $P,$ one pseudo-metric for each boundary point in $ \partial \mathbb{D}^{2}.$ This is done in Section 2.

In Section 3 we show that the same construction works for proper uniquely geodesic metric spaces homeomorphic to $\mathbb{R}^2 .$ In Section 4 we show that when  $\mathcal{C}$ is just the class of the straight lines in $\mathbb{D}^2 ,$ the construction of the metric $d$ on  $P$ gives rise to a Hilbert geometry, that is, straight lines are precisely the geodesics with respect to $d, $ and yet, $(P,d)$ as a metric space is not isometric to any convex domain in $\mathbb{R} ^{2}$ equipped with the standard Hilbert metric defined via cross ratio.

\section{Definitions and Preliminaries}

Let $\mathbb{D}^{2}$ be the (open) unit disk in $\mathbb{R}^{2}$ and let $%
\mathcal{C}$ be a family of continuous and injective maps $I\rightarrow\mathbb{D}^{2}\cup
\partial \mathbb{D}^{2}$ where $I$ is a close interval in $\mathbb{R},$  the endpoints of $I$ are mapped
in $ \partial \mathbb{D}^{2}$ and the interior of $I$ is mapped into $\mathbb{D}^{2}.$
Assume that $\mathcal{C}$  satisfies the following properties:

\begin{enumerate}
\item[(C1)] for any two points $x,y\in \mathbb{D}^{2}$ there exists a unique
curve $c_{xy}\in \mathcal{C}$ containing both $x$ and $y.$ The restriction
of $c_{xy}$ with endpoints $x$ and $y$ will be called the segment from $x$
to $y$ and will be denoted by $\left[ x,y\right] .$

\item[(C2)] for any two points $\xi ,\eta \in \partial \mathbb{D}^{2}$ there
exists a unique curve $c_{\xi \eta }\in \mathcal{C}$ with endpoints $\xi $
and $\eta .$ We call $c_{\xi \eta }$ the line from $\xi $ to $\eta $ and
denote it by $\left( \xi ,\eta \right) .$

\item[(C3)] for any $x\in $ $\mathbb{D}^{2}$ and $\xi \in \partial \mathbb{D}%
^{2}$ there exists a unique curve $c_{x\xi }\in \mathcal{C}$ having $\xi $
as one endpoint and containing $x.$ The restriction of $c_{x\xi }$ with
endpoints $x$ and $\xi $ will be called the ray from $x$ to $\xi $ and will
be denoted by $\left[ x,\xi \right) .$
\end{enumerate}

Examples of such families include the geodesic lines in the hyperbolic disk as well as homeomorphic images of these.

From the above properties it easily follows that for any two curves $%
c,c^{\prime }\in \mathcal{C}$ either $c\cap c^{\prime }=\varnothing $ or, $%
c\cap c^{\prime }$ is a singleton. In particular, the same holds for any two
segments. 

Let $\theta _{1},\theta _{2},\theta _{3},\theta _{4}, \theta_1$ be four points in $%
\partial \mathbb{D}^{2}$ in cyclic clockwise order. Each pair of points $%
\theta _{i},\theta _{j}$ $i\neq j$ determines exactly two subarcs of $%
\partial \mathbb{D}^{2}.$ Denote by

\begin{description}
\item $\wideparen{\theta_{1}\theta_{2}} \equiv \wideparen{\theta_{2}\theta_{1}} $  the subarc not containing  $\theta_{3},\theta_{4},$

\item $\wideparen{\theta_{3}\theta_{4}} \equiv \wideparen{\theta_{4}\theta_{3}} $  the subarc not containing $\theta_{1},\theta_{2},$

\item $\wideparen{\theta_{2}\theta_{4}} \equiv \wideparen{\theta_{4}\theta_{2}} $ the subarc not containing $\theta_{1},$

\item $\wideparen{\theta_{1}\theta_{3}} \equiv \wideparen{\theta_{3}\theta_{1}} $ the subarc not containing $\theta_{4}.$
\end{description}

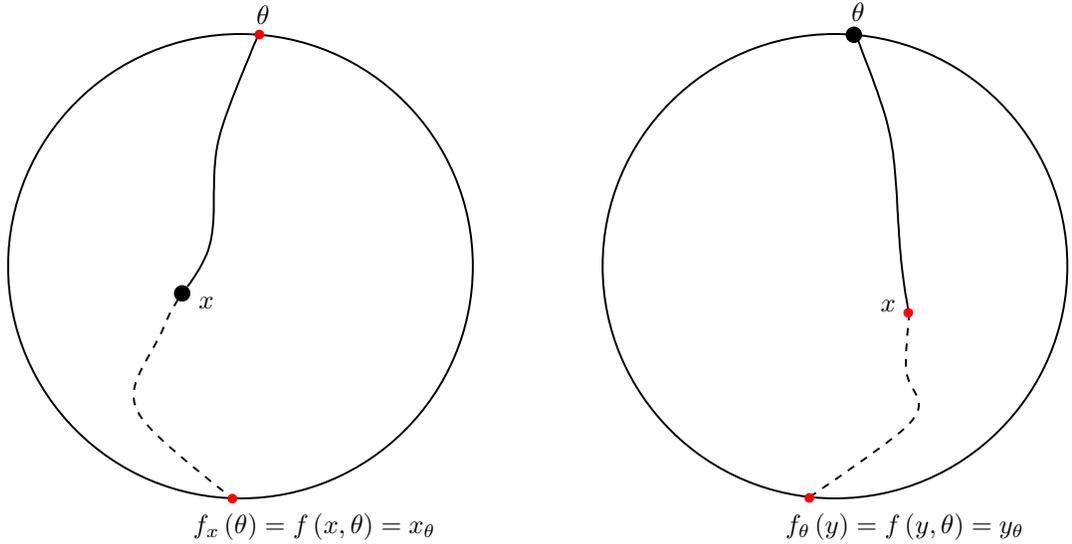
\begin{figure}[ptb]
\begin{center}
\includegraphics[scale=1.55]{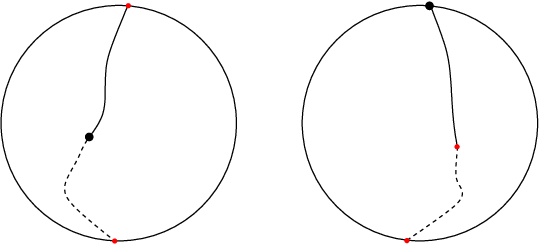}
\hspace*{14.3cm}\begin{picture}(0,0)
\put(-87,195){$\theta$}
\put(-312,194){$\theta$}
\put(-334,86){$x$}
\put(-76,85){$x$}
\put(-336,1){$f_x \left( \theta \right) = f \left( x, \theta \right) = x_{\theta}$}
\put(-112,1){$f_{\theta} \left( y \right) = f \left( y, \theta \right) = y_{\theta}$}
\end{picture}\\[7mm]
\end{center}
\par
\label{figdef1}
\caption{The projection maps $f_{x}:\partial \mathbb{D}^{2}\rightarrow \partial \mathbb{D}^{2}$ and $f_{\theta }:\mathbb{D}^{2}\rightarrow \partial \mathbb{D}^{2}.$ }
\end{figure}

In the sequel we will also deal with arcs determined by a triple of points  $\theta _{1},\theta _{2},\theta _{3}$ in $
\partial \mathbb{D}^{2}.$ In this case $\wideparen{\theta_{1}\theta_{2}} \equiv \wideparen{\theta_{2}\theta_{1}} $ is  the subarc not containing  $\theta_{3}$  and similarly for $\wideparen{\theta_{1}\theta_{3}},$  $\wideparen{\theta_{2}\theta_{3}}.$

Clearly, each $c\in \mathcal{C}$ splits $\mathbb{D}^{2}$ into two
components. We will say that two curves $c,c^{\prime }\in \mathcal{C}$
intersect transversely (by (C1), at a single point)  if $c$  intersects both components of $
\mathbb{D}^{2}$ determined by $c^{\prime }.$ The following property follows from
(C1)-(C3): 

\begin{enumerate}
\item[(C4)] all intersections between curves in $\mathcal{C}$ are transverse.
\end{enumerate}

To see this assume, on the contrary, that the lines $\left( \xi ,\xi
^{\prime }\right) $ and $\left( \eta ,\eta ^{\prime }\right) $ intersect at
a (single) point $x\in \mathbb{D}^{2}$ and the intersection is not
transverse. We may assume that the cyclic clockwise order of the boundary points of these
lines is $\eta ,\xi ,\xi ^{\prime },\eta ^{\prime },\eta .$
For any points $\theta \in \wideparen{\eta \xi }$ and $\theta ^{\prime }\in
\wideparen{\eta ^{\prime }\xi ^{\prime }}$ we claim that the line $\left( \theta ,\theta
^{\prime }\right) $ necessarily contains $x.$ To check this observe that the union $c_{x\eta}\cup c_{x\xi}$ separates $\mathbb{D}^{2}\cup\partial \mathbb{D}^{2}$ into two components, one containing $\theta$ and the other $\theta ^{\prime }.$ If $x\notin \left( \theta ,\theta^{\prime }\right)$ then $\left( \theta ,\theta^{\prime }\right)$ must intersect transversely the interior of either $c_{x\eta}$ or $ c_{x\xi}.$
 Without loss of generality, assume that $\left( \theta ,\theta^{\prime }\right)$ intersects transversely the interior of $c_{x\eta}$ at a point, say, $z_1 .$ The line $\left( \eta ,\eta ^{\prime }\right) $ separates $\mathbb{D}^{2}\cup\partial \mathbb{D}^{2}$ into two components and both  $ \theta ,\theta^{\prime }$ belong to one of them. It follows that  $\left( \theta ,\theta^{\prime }\right)$ must intersect $\left( \eta ,\eta ^{\prime }\right) $ at an other point, say, $z_2.$ Then the points $z_1 ,z_2$ belong to  $\left( \theta ,\theta^{\prime }\right)$ as well as to $\left( \eta ,\eta ^{\prime }\right) $ which violates property (C1). This shows that $x\in \left( \theta ,\theta^{\prime }\right).$\\
 To complete the proof of (C4) consider a point $\theta ^{\prime \prime }\in \wideparen{\eta
^{\prime }\xi ^{\prime }}$ with $\theta ^{\prime \prime } \neq \theta ^{\prime }.$ Then we would have  $x\in \left( \theta ,\theta ^{\prime
}\right) \cap $ $\left( \theta ,\theta ^{\prime \prime }\right) $ and, by  property (C3), $\left( \theta ,\theta ^{\prime }\right) \equiv\left( \theta ,\theta ^{\prime \prime }\right) $ which, by (C2), implies that $\theta ^{\prime \prime } = \theta ^{\prime },$ a contradiction.

We will also need the following property which follows immediately from property (C3):
\begin{enumerate}
\item[(C5)] for pairwise distinct points $\theta, \theta_1 , \ldots \theta_{m} $ in $\partial \mathbb{D}^2 ,$
$\left( \theta,\theta_i \right) \cap \left( \theta ,\theta_j \right) =\emptyset$ if $i\neq j.$
\end{enumerate}

Denote by $\left\vert \wideparen{\theta _{i}\theta _{j}}\right\vert $ the
Euclidean length of the arc $\wideparen{\theta _{i}\theta _{j}}$ and define
the ratio $\left[ \theta _{1},\theta _{2},\theta _{3},\theta _{4}\right] $
of the points $\theta _{1},\theta _{2},\theta _{3},\theta _{4}$ by 
\begin{equation}
\left[ \theta _{1},\theta _{2},\theta _{3},\theta _{4}\right] :=\frac{%
\left\vert \wideparen{\theta _{1}\theta _{3}}\right\vert \left\vert \wideparen{%
\theta _{4}\theta _{2}}\right\vert }{\left\vert \wideparen{\theta _{1}\theta
_{2}}\right\vert \left\vert \wideparen{\theta _{4}\theta _{3}}\right\vert }.
\label{starTHETAA}
\end{equation}

We will say that $X$ is a convex subset of $\mathbb{D}^{2}$ if for any two
points $x,y\in X$ the segment $\left[ x,y\right] $ is entirely contained in $%
X.$ We say that $P$ is a convex polygon in $\mathbb{D}^{2}$ if $P$ is an
open convex subset of $\mathbb{D}^{2}$ whose boundary is a finite union of
segments.

\begin{figure}[ptb]
\begin{center}
\includegraphics[scale=2.2]{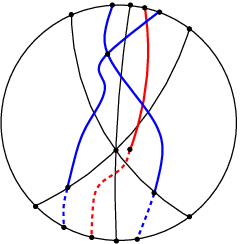}
\begin{picture}(0,0)
\put(-122,257){$\overline{\theta}$}
\put(-107,254){$ \theta_{n}$}
\put(-146,259){$ \xi_{2}$}\put(-90,250){$ \xi_{1}$}
\put(-60,233){$ \eta_{1}$}\put(-189,248){$ \eta_{2}$}
\put(-154,199){$ z $}
\put(-147,96){$\overline{x}$} \put(-117,97){$x_n$}
\put(-90,54){$ z_{1}$}\put(-198,62){$ z_{2}$}
\put(-53,20){$ \overline{x}_{\eta_2} $}\put(-235,34){$ \overline{x}_{\eta_1} $}
 \put(-200,8){$ z_{\xi_1}$}

\put(-139,-8){$\overline{x}_{\overline{\theta}}$}
\put(-116,-6){$z_{\xi_2}$}
\put(-168,-2){${x_{n}}_{\theta_{n}} $}
\end{picture}\\[7mm]
\end{center}
\par
\label{thetax}
\caption{Notation for the proof of Lemma \ref{contin}.  }
\end{figure}
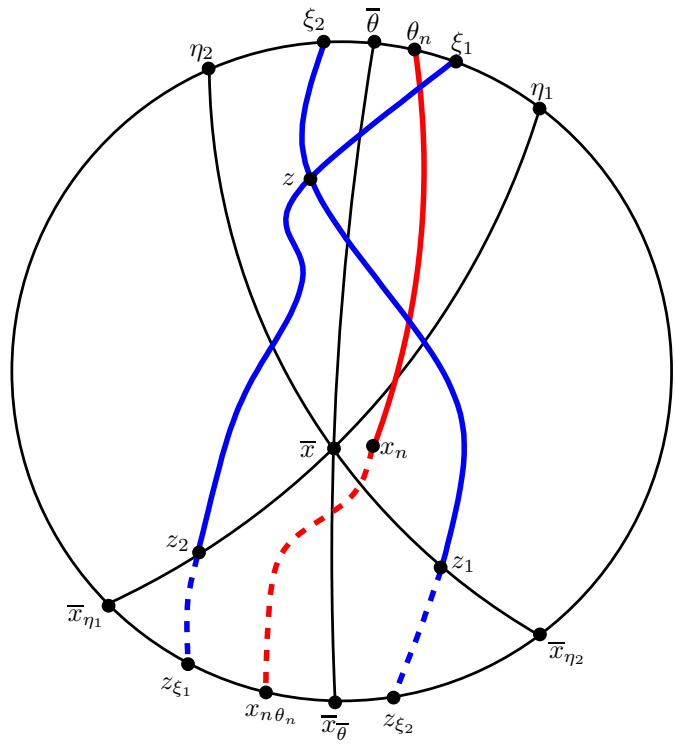

\section{A metric with prescribed geodesics\label{sectionPrescribed}}

Our goal is to define a metric $d$ on a convex polygon $P$ such that the
metric space $(P,d)$ is a geodesically complete metric space whose geodesics
are precisely $\left\{ c\cap P\bigm\vert c\in \mathcal{C}\right\} .$ We note here that the word geodesic
has the classical meaning, namely, the image of an isometric map $\mathbb{R}\rightarrow P.$ \\
We will
be using the terminology segment, ray and line as introduced in properties
(C1), (C2) and (C3).

Define a function
\begin{equation*}
f:\mathbb{D}^{2}\times \partial \mathbb{D}^{2}\rightarrow \partial \mathbb{D}%
^{2}
\end{equation*}%
as follows: for $(x,\theta )\in \mathbb{D}^{2}\times \partial \mathbb{D}^{2}$
there exists, by property (C3), a unique curve $c_{x\theta }$ containing $x$
and having $\theta $ as one endpoint. Set $f\left( x,\theta \right) $ to be
the other endpoint of $c_{x\theta }.$ \newline
For fixed $x\in \mathbb{D}^{2}$ we denote by $f_{x}$ the induced map
\begin{equation*}
f_{x}:\partial \mathbb{D}^{2}\rightarrow \partial \mathbb{D}^{2}\mathrm{\
given\ by\ }f_{x}(\theta ):=f(x,\theta ),
\end{equation*}
see Figure 1. Similarly, for fixed $\theta \in \partial \mathbb{D}^{2}$ we have the
induced map
\begin{equation*}
f_{\theta }:\mathbb{D}^{2}\rightarrow \partial \mathbb{D}^{2}\mathrm{\
given\ by\ }f_{\theta }(x):=f(x,\theta ).
\end{equation*}%
%

\begin{notation}
We will be writing $x_{\theta }$ instead of $f\left( x,\theta \right) $ and
we will be calling $x_{\theta }$ the projection from $\theta $ of $x$ to the boundary $
\partial \mathbb{D}^{2} .$
\end{notation}

\begin{lemma}
\label{contin} The function $f: \mathbb{D}^{2}\times \partial \mathbb{D}%
^{2}\rightarrow\partial \mathbb{D}^{2}$ defined above is continuous.
\end{lemma}

\begin{proof} For the reader's convenience, all  notation introduced in this proof is included in  Figure 2.

Assume $f$ is not continuous at a point $(\overline{x},\overline{\theta }%
)\in \mathbb{D}^{2}\times \partial \mathbb{D}^{2}.$ Then, there must exist a
positive real $\varepsilon _{0}$ such that
\begin{equation}
\forall n\in \mathbb{N},\exists \left( x_{n},\theta _{n}\right) \mathrm{\
satisfying\ }\left\vert \wideparen{ \overline{\theta}\, \theta _{n}}%
\right\vert <\frac{1}{n},\left\vert \overline{x}-x_{n}\right\vert <\frac{1}{n%
}\mathrm{\ and\ }\left\vert \wideparen{\overline{x}_{\overline{\theta }}\,
{x_n}_{\theta _{n}} }\right\vert \geq \varepsilon _{0}.  \label{epsilonNOT}
\end{equation}
where the arc $ \wideparen{ \overline{\theta}\, \theta _{n}} $ is the subarc of $\partial\mathbb{D}^2$ not containing $ \overline{x}_{\overline{\theta }}$ and
$ \wideparen{\overline{x}_{\overline{\theta }}\,{x_n}_{\theta _{n}} }$ the subarc not containing $\overline{\theta}.$

There are exactly two boundary points in $\partial \mathbb{D}^{2}$ each of
which determines, along with $\overline{x}_{{\overline{\theta }}},$ an arc of
Euclidean length $\varepsilon _{0}.$ In other words, there exist $\eta
_{1},\eta _{2}$ in $\partial \mathbb{D}^{2}$ such that the projections $%
\overline{x}_{\eta _{1}},\overline{x}_{\eta _{2}}$ satisfy
\begin{equation}
\left\vert \wideparen{ \overline{x}_{{\overline{\theta }}} \,
\overline{x}_{\eta_1} }\right\vert =\varepsilon _{0}=\left\vert \wideparen{
\overline{x}_{{\overline{\theta }}} \, \overline{x}_{\eta_2} }\right\vert
\label{epsilonNOTNOT}
\end{equation}%
where,  $ \wideparen{ \overline{x}_{{\overline{\theta }}} \, \overline{x}_{\eta_1} } ,
\wideparen{\overline{x}_{{\overline{\theta }}} \, \overline{x}_{\eta_2} }
$ are the subarcs not containing $\overline{\theta}.$
Let $\wideparen{\eta_1 \, \eta_2}$ be the subarc of $\partial \mathbb{D}^{2}$ which has endpoints $\eta_1 \, \eta_2$ and does not contain $\overline{x}_{\eta_1}, \overline{x}_{\eta_2}.$ Clearly, by transversality of the intersection of the lines $c_{\overline{x}%
\,\overline{\theta }},c_{\overline{x}\,\eta _{1}},c_{\overline{x}\,\eta
_{2}} $ the arc $\wideparen{\eta_1 \, \eta_2}$ contains $\overline{\theta }.$

Let $ \wideparen{ \overline{\theta}\, \eta_1}$ be the subarc not containing $\overline{x}_{\eta_1}, \overline{x}_{\eta_2} $ and similarly we specify $ \wideparen{ \overline{\theta}\, \eta_2}.$
Pick arbitrary points $\xi _{1}\in \wideparen{ \overline{\theta}\, \eta_1},$   $\xi
_{2}\in \wideparen{ \overline{\theta}\, \eta_2},$  $z_1 \in  \left[\overline{x} ,\overline{x}_{\eta_2}\right) $ and $z_2 \in  \left[\overline{x} ,\overline{x}_{\eta_1}\right) $ with $z_1 \neq \overline{x} \neq z_2 .$
The
rays $\left[ z_{1},\xi _{2}\right) $ and $\left[ z_{2},\xi _{1}\right) $
intersect at a point, say, $z$ and $\overline{x}$ is contained in the region
$R_{z}$ bounded by the rays $\left[ z,z_{\xi _{1}}\right) ,\left[ z,z_{\xi
_{2}}\right) $ and the arc $\wideparen{z_{\xi_1} \, z_{\xi_2}}$ where the latter is the subarc not containing $\xi_1 , \xi_2.$

Without loss of generality and by choosing, if necessary, a subsequence we
may assume that the sequence $\left\{ \theta _{n}\right\} \subseteq %
\wideparen{ \xi_1\, \xi_2}$ and the sequence $\left\{ x_{n}\right\} $ is
contained in the region $R_{z}.$ Now for any $\left( x_{n},\theta
_{n}\right) \in \wideparen{ \xi_1\, \xi_2}\times R_{z}$ we have that $%
c_{x_{n}\theta _{n}}$ intersects both lines $\left( \xi _{1},z_{\xi
_{1}}\right) ,\left( \xi _{2},z_{\xi _{2}}\right) $ at exactly one point and, thus, the
projection $x_{n}{}_{\theta _{n}}$ is contained in $%
\wideparen{z_{\xi_1} \,
z_{\xi_2}}.$ As $\wideparen{z_{\xi_1} \, z_{\xi_2}}\subset
\wideparen{
\overline{x}_{\eta_1}   \, \overline{x}_{\eta_2}   }$ we have, by (\ref
{epsilonNOTNOT}), $\left\vert
\wideparen{  \overline{x}_{{\overline{\theta
}}} \, {x_{n}}_{\theta_{n}}   }\right\vert <\varepsilon _{0}$ which
contradicts (\ref{epsilonNOT}).
\end{proof}

\begin{figure}[ptb]
\begin{center}
\includegraphics[scale=1.15]{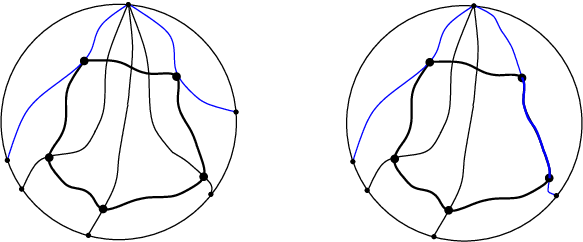}
\begin{picture}(0,0)
\put(-259,134){$\theta$}\put(-68,134){$\theta$}
\put(-375,40){$g_1(\theta)=\theta_1$}
\put(-193,70){$\theta_2= g_2(\theta)$}
\put(-184,42){$ g_1(\theta)=\theta_1$}
\put(-20,14){$\theta_2 =\theta_3 = g_2(\theta)$}
\put(-322,18){$\theta_k$}\put(-129,18){$\theta_k$}
\put(-286,-7){$\theta_{k-1}$}\put(-94,-7){$\theta_{k-1}$}
\put(-208,17){$\theta_3$}
\put(-289,104){$v_1$}\put(-97,104){$v_1$}
\put(-225,91){$v_2$}\put(-32,91){$v_2$}
\put(-212,38){$v_3$}\put(-20,37){$v_3$}
\put(-309,51){$v_k$}\put(-117,51){$v_k$}
\end{picture}\\[7mm]
\end{center}
\par
\label{figdef3}
\caption{The functions $g_1 ,g_2 :$ on the right is the case where the line $ \left( \theta, \theta_2\right)$ contains the vertex $v_3$ and, hence, the side $[v_2,v_3]$}
\end{figure}
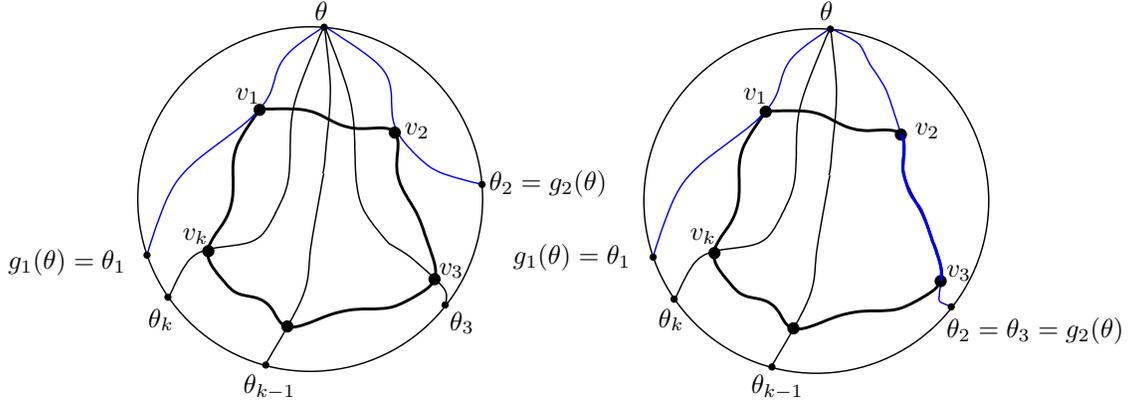

Denote by $v_i, i=1,\ldots , k$ the vertices of $P$ and fix a point $\theta \in \partial \mathbb{D}^{2}.$
For each $i$ denote by $\left( \theta, \theta_i \right) $ the line containing $v_i$ (see Figure 3). Observe that for $i\neq j$ it may happen that the line $\left( \theta, \theta_i \right) $ contains a vertex $v_j, i\neq j$ in which case $\left( \theta, \theta_i \right) =\left( \theta, \theta_j \right) .$ Otherwise, by property (C5),
$\left( \theta, \theta_i \right) \cap \left( \theta, \theta_j \right) =\emptyset. $ In particular, the above lines are at most $k$ and at least $2.$ Since each $ \left( \theta, \theta_i \right)$ separates $ \mathbb{D}^{2}$
into two components there exist exactly two of them, say,  $ \left( \theta, \theta_1 \right)$ and $ \left( \theta, \theta_2 \right)$ which are outermost in the following sense: one component of $ \mathbb{D}^{2}\setminus \left( \theta, \theta_1 \right)$ contains all lines $ \left( \theta, \theta_i\right), i\neq 1$ and the other component contains none and similarly for $ \left( \theta, \theta_2 \right).$ Clearly, the line $ \left( \theta, \theta_1\right)$ either contains only the vertex $v_1$ or, by convexity of $P,$ contains a segment
$\left[ v_1 ,v_1^{\prime}\right] .$ In the former case, $\theta_1$ is  just  $f_{v_1} (\theta).$ The same holds for  $ \left( \theta, \theta_2 \right).$

The above discussion shows that for each $\theta \in \partial \mathbb{D}^{2}$ there exist  unique points $g_{1}\left( \theta\right) ,g_{2}\left( \theta \right) $ in $\partial\mathbb{D}^{2}$  such that the region bounded by the lines $\left( \theta,g_{1}\left( \theta \right) \right) ,$ $\left( \theta ,g_{2}\left( \theta
\right) \right) $ and the arc $\wideparen{ g_{1}\left( \theta\right)
g_{2}\left( \theta\right) }$ contains $P$ and is minimal with respect to the
property of containing $P.$ The arc $\wideparen{ g_{1}\left( \theta\right)g_{2}\left( \theta\right) }$ is meant to be the subarc with endpoints $g_{1}\left( \theta\right), g_{2}\left( \theta\right) $ not containing $\theta.$

For simplicity, we will be writing $\xi _{\theta
}$ and $\eta _{\theta }$ instead of $g_{1}\left( \theta \right) $ and $%
g_{2}\left( \theta \right) $ respectively.

\begin{figure}[ptb]
\begin{center}
\includegraphics[scale=1.55]{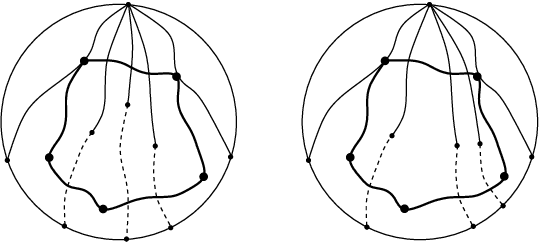}
\hspace*{14.2cm}
\begin{picture}(0,0)
\put(-87,193){$\theta$}
\put(-312,193){$\theta$}
\put(-188,65){$ \xi_{\theta}$}
\put(-412,65){$ \xi_{\theta}$}
\put(-5,70){$ \eta_{\theta}$}
\put(-230,70){$ \eta_{\theta}$}
\put(-53,14){$y_{\theta}$}
\put(-283,14){$y_{\theta}$}
\put(-143,16){$x_{\theta}$}
\put(-370,16){$x_{\theta}$}
\put(-122,92){$x$}
\put(-347,92){$x$}
\put(-74,82){$y$}
\put(-300,82){$y$}
\put(-44,83){$ z $}
\put(-26,33){$ z_{\theta}$}
\put(-321,6){$ z_{\theta}$}
\put(-308,113){$ z $}
\put(-99,-17){Case B}\put(-315,-17){Case  A}
\end{picture}\\[7mm]
\end{center}
\par
\label{figdef4}
\caption{The arcs defining the cross  ratio $\left[ \protect\eta_{\protect\theta
},y_{\protect\theta},x_{\protect\theta},\protect\xi_{\protect\theta}\right] $}
\end{figure}
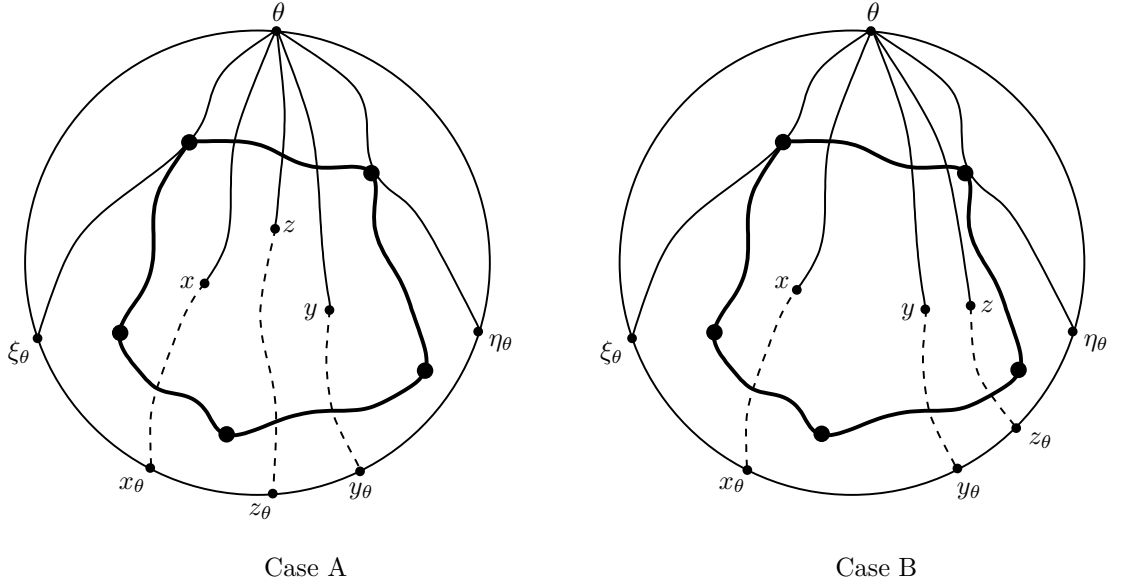

\begin{lemma}
\label{continn} The maps $g_{i}:\partial \mathbb{D}^{2}\rightarrow \partial
\mathbb{D}^{2},i=1,2$ are continuous.
\end{lemma}

\begin{proof}
If $\theta $ is not the endpoint of a line containing a segment of $\partial
P,$ then, as explained above,  $g_{1}\left( \theta \right) $ is given by some $f_{v_1}$ for a vertex
$v_1 .$ Therefore, $g_{1}$ is continuous at every such $\theta .$ Let now $%
\theta $ be the endpoint of a line containing a segment $\left[ v,w\right] $
of $\partial P$ and $\wideparen{\theta ^{+}\theta
^{-}}$ a sufficiently small neighborhood of $\theta $ in $\partial \mathbb{D}%
^{2}.$ Clearly, $f_{v}\left( \theta \right) =f_{w}\left( \theta \right) $
and either $g_{1}=f_{v}$ on $\wideparen{\theta ^{+}\theta }$ and $%
g_{1}=f_{w} $ on $\wideparen{\theta ^{-}\theta }$ or, $g_{1}=f_{w}$ on $%
\wideparen{\theta ^{+}\theta }$ and $g_{1}=f_{v}$ on $%
\wideparen{\theta
^{-}\theta }.$ In any case $g_{1}$ is continuous at every $\theta \in
\partial \mathbb{D}^{2}$ and identically $g_{2}$ is.
\end{proof}

For any point $\theta \in \partial \mathbb{D}^{2}$ we first define a
pseudo-metric $d_{\theta }$ on $P.$ Given any two points $x,y\in P$ we may
assume, without loss of generality, that the clockwise cyclic order of the
four images of $\theta $ under the functions $f_{x},f_{y}$ and $g_{1},g_{2}$
is $g_{2}\left( \theta \right) ,f_{y}\left( \theta \right) ,f_{x}\left(
\theta \right) ,g_{1}\left( \theta \right) .$ With the simplified notation
introduced above these four points are: $\eta _{\theta },y_{\theta
},x_{\theta },\xi _{\theta }.$ Define the $\theta -$distance of the points $%
x,y$ by
\begin{equation}
d_{\theta }\left( x,y\right) :=\log \left[ \eta _{\theta },y_{\theta
},x_{\theta },\xi _{\theta }\right] =
\log \frac{\left \vert \wideparen{\eta_{\theta}x_{\theta}}\right \vert
\left \vert \wideparen{\xi_{\theta}y_{\theta}}\right \vert }{\left \vert %
\wideparen{\eta_{\theta}y_{\theta}}\right \vert \left \vert %
\wideparen{\xi_{\theta }x_{\theta}}\right \vert }.
\label{starTHETA}
\end{equation}

\begin{lemma}
\label{dTHETA} The distance $d_{\theta}$ is a pseudo-metric on $P.$
\end{lemma}

\begin{proof}
Clearly $d_{\theta}$ is symmetric. Moreover, if $x\neq y$ then for every $%
\theta$ which is not an endpoint of the curve $c_{xy}$ we have $%
x_{\theta}\neq y_{\theta}$ which implies
\begin{equation*}
\frac{\left \vert \wideparen{\eta_{\theta}x_{\theta}}\right \vert }{\left
\vert \wideparen{\eta_{\theta}y_{\theta}}\right \vert }>1\mathrm{\ and\ }%
\frac{\left \vert \wideparen{\xi_{\theta}y_{\theta}}\right \vert }{\left
\vert \wideparen{\xi_{\theta}x_{\theta}}\right \vert }>1\Longrightarrow
d_{\theta}\left( x,y\right) = \log \frac{\left \vert \wideparen{\eta_{%
\theta}x_{\theta}}\right \vert \left \vert \wideparen
{\xi_{\theta}y_{\theta}}\right \vert }{\left \vert \wideparen{\eta
_{\theta}y_{\theta}}\right \vert \left \vert \wideparen{\xi_{\theta}x_{%
\theta}}\right \vert }\gvertneqq \log1=0.
\end{equation*}
This shows that $d_{\theta}\left( x,y\right) \neq0$ if $x\neq y$ and $\theta$
is not an endpoint of $c_{xy}.$ Clearly,
\begin{equation}
  \textrm{\ if\ } \theta \textrm{\ is\ an\ endpoint\ of\ } c_{xy}
   \textrm{\ then\ } d_{\theta}\left( z,w\right) = 0  \textrm{\ for\ any\ pair\ of\ points\ }
z,w\in c_{xy}.
\label{xytheta}
\end{equation}
This is the reason $d_{\theta}$ is just a pseudo-metric and
not a metric. \newline
For the triangle inequality,  let $x,y,z\in P$ and observe that in the definition of $d_{\theta}$ the points involved are projected onto the boundary circle where the $\log$ of the cross ratio obeys the triangle inequality. We carry out the calculation by considering two cases (see Figure 4): 

\begin{description}
\item[Case A] $z_{\theta}\in \wideparen{x_{\theta}y_{\theta}}.$ In this case
the triangle inequality is, in fact, equality:
\begin{align}
d_{\theta}\left( x,z\right) +d_{\theta}\left( z,y\right) &= \log \left[
\eta_{\theta},z_{\theta},x_{\theta},\xi_{\theta}\right] + \log \left[
\eta_{\theta},z_{\theta},y_{\theta},\xi_{\theta}\right]  \notag \\
& =\log \left( \frac{\left \vert \wideparen{\eta_{\theta}x_{\theta}}\right
\vert\left \vert \wideparen{\xi_{\theta}z_{\theta}}\right \vert } {\left
\vert\wideparen{\eta_{\theta}z_{\theta}}\right \vert \left \vert %
\wideparen{\xi_{\theta}x_{\theta}}\right \vert } \cdot \frac{\left \vert %
\wideparen{\eta_{\theta}z_{\theta}}\right \vert \left \vert %
\wideparen{\xi_{\theta}y_{\theta}}\right \vert} {\left \vert %
\wideparen{\eta_{\theta}y_{\theta}}\right \vert \left \vert%
\wideparen{\xi_{\theta}z_{\theta}}\right \vert } \right)  \notag \\
&=\log \frac{\left \vert \wideparen{\eta_{\theta}x_{\theta}}\right \vert
\left \vert \wideparen{\xi_{\theta}y_{\theta}}\right \vert }{\left \vert %
\wideparen{\eta_{\theta}y_{\theta}}\right \vert \left \vert %
\wideparen{\xi_{\theta }x_{\theta}}\right \vert } =d_{\theta}\left(
x,y\right)  \label{triangleequality}
\end{align}

\item[Case B] $z_{\theta}\notin \wideparen{x_{\theta}y_{\theta}}.$ We may
assume that $z_{\theta}$ is contained in the interior of the arc $%
\wideparen{y_{\theta
}\eta_{\theta}}$ (the case $z_{\theta}\in \wideparen{\xi_{\theta}x_{\theta}}$
is treated in an identical manner) which implies that $y_{\theta}\in %
\wideparen{x_{\theta}z_{\theta}}.$ By Case A we have
\begin{equation*}
d\left(x,y\right) + d\left( y,z\right) =d\left( x,z\right) \Longrightarrow
d\left(x,y\right) \leq d\left( x,z\right)\leq d\left( x,z\right) +d\left(
z,y\right) .
\end{equation*}
\end{description}
\end{proof}

We now define a metric $d$ on $P.$

\begin{definition}
\label{defd}Consider a countable dense subset $\Theta =\left\{ \theta
_{i}|i\in \mathbb{N}\right\} $ of $\partial \mathbb{D}^{2}$ with the
following property: for each segment $\left[ v,w\right] $ in $\partial P,$
the endpoints of the curve $c_{vw}$ are contained in $\Theta .$ To each $%
\theta _{i}$ in $\Theta $ assign a positive real $w_{i}$ such that the
series $\sum\limits_{i}w_{i}$ converges. For $x,y$ $\in P$ define
\begin{equation*}
d\left( x,y\right) :=\sum\limits_{i=1}^{\infty }w_{i}\,d_{\theta _{i}}\left(
x,y\right) .
\end{equation*}%
Observe that for fixed $x,y$ $\in P$ the function $\theta \rightarrow
d_{\theta }\left( x,y\right) $ is, by Lemmata \ref{contin} and \ref{continn}%
, continuous on $\partial \mathbb{D}^{2}.$ It follows that the set $\left\{
d_{\theta _{i}}\left( x,y\right) \bigm \vert\theta _{i}\in \Theta \right\} $
is, by compactness of $\partial \mathbb{D}^{2},$ bounded by some $M>0$ and
thus%
\begin{equation*}
d\left( x,y\right) \leq M\sum\limits_{i=1}^{\infty }w_{i}<+\infty .
\end{equation*}
\end{definition}

\begin{proposition}\label{6one} $d$ is a metric on $P.$ \end{proposition}
\begin{proof}
 The triangle inequality for $d$ follows from the triangle inequality of all $%
d_{\theta}$ proven in Lemma \ref{dTHETA}. Similarly, if $x\neq y$ then $%
d_{\theta}\left( x,y\right) \neq0$ for every $\theta\in \Theta$ which is not
an endpoint of the curve $c_{xy}$ (see the beginning of the proof of Lemma %
\ref{dTHETA}). Therefore, $d$ is a metric on $P.$
\end{proof}

\begin{proposition}\label{equivTOP}
The topology induced by $d$ is equivalent to the Euclidean topology.
\end{proposition}
\begin{proof}
 We first show that the Euclidean topology of $P$ is thinner than the topology induced by $d. $ For this, it suffices to show that \begin{equation}
\forall z_{0}\in P\mathrm{\ and\ }\varepsilon >0,\exists \rho >0\mathrm{\
such\ that\ }\left\vert z-z_{0}\right\vert <\rho \Rightarrow
d(z,z_{0})<\varepsilon .  \label{thinner1}
\end{equation}
Observe that, by Lemma \ref{contin}, property (\ref{thinner1}) holds
for the pseudo-metric $d_{\theta },$ namely, for fixed $\theta \in \partial
\mathbb{D}^{2}$
\begin{equation}
\exists \rho _{\theta }>0\mathrm{\ such\ that\ }\left\vert
z-z_{0}\right\vert <\rho _{\theta }\Rightarrow d_{\theta
}(z,z_{0})<\varepsilon .  \label{thinner}
\end{equation}
Let $N(z_{0})$ be a compact (in the Euclidean topology) neighborhood in $P$
containing $z_{0}.$ Let $F:\partial \mathbb{D}^{2}\times N(z_{0})\times
N(z_{0})\rightarrow \mathbb{R}$ be the function given by
\begin{equation*}
F(\theta ,w,z)=d_{\theta }(w,z)=\log \left( \frac{\left\vert %
\wideparen{\eta_{\theta}z_{\theta}}\right\vert }{\left\vert %
\wideparen{\eta_{\theta}w_{\theta}}\right\vert }\frac{\left\vert %
\wideparen{\xi_{\theta}w_{\theta}}\right\vert }{\left\vert %
\wideparen{\xi_{\theta}z_{\theta}}\right\vert }\right) .
\end{equation*}%
By Lemma \ref{continn}, the projection points $g_{1}(\theta )=\xi _{\theta }$
and $g_{2}(\theta )=\eta _{\theta }$ depend continuously on $\theta $ and by
Lemma \ref{contin} the same holds for every projection point $w_{\theta }.$
This shows that $F$ is continuous on $\partial\mathbb{D}^{2}\times N(z_{0})\times N(z_{0})$.
In particular,  (\ref{thinner}) holds. By compactness, $F$ is
uniformly continuous which implies that $\rho _{\theta }$ in (\ref{thinner})
can be chosen independent of $\theta .$ It follows that (\ref{thinner1})
holds.

We proceed with the proof of the proposition by showing the converse. For this it suffices to show that for any sequence
$\left\{ z_n \right\}$ converging to $z_0 $ with respect to the metric $d,$ we have $ \left\vert z_n -z_0\right\vert \rightarrow 0.$ Assume, on the contrary, that $\left\{ z_n \right\}$ is a sequence in $P$ with $d\left( z_0, z_n \right)\rightarrow0$ and $ \left\vert z_n -z_0\right\vert > \epsilon_0$ for some $\epsilon_0 >0.$

By choosing, if necessary, a subsequence we may assume that $\left\{ z_n \right\}$ converges (in the Euclidean sense) to a point $z^{\prime} .$ Let $\theta_0$ be a point in $\partial \mathbb{D}^2$ such that the line $c_{z_0 \theta_0} $ does not contain $z^{\prime} .$ Pick a compact Euclidean ball $B(z^{\prime})$ containing $z^{\prime}$ such that
\begin{equation} B(z^{\prime})\cap c_{z_0 \theta_0} =\emptyset . \label{mistake}\end{equation}
The image  $f_{z_0} \bigl( B(z^{\prime}) \bigr)$ of $B(z^{\prime})$ under the continuous map $f_{z_0}: \partial \mathbb{D}^2\rightarrow\partial \mathbb{D}^2$ is a compact and connected subset of $\partial \mathbb{D}^2 $ which, by (\ref{mistake}),    does not contain $\theta_0 .$ It follows that we may pick a compact subinterval $N\left( \theta_0 \right)$ of $\partial \mathbb{D}^2 $ containing $\theta_0$ and disjoint from
$f_{z_0} \bigl( B(z^{\prime}) \bigr) .$ By (\ref{xytheta}) we know that for any $\theta \in \partial \mathbb{D}^2 $ and $ x\in\mathbb{D}^2 $
\[  d_{\theta}(x,z_0)=0 \Longleftrightarrow x\in c_{z_0 \theta} \]
which implies that
\begin{equation}
 d_{\theta}(z,z_0)>0 \textrm{\ for\ all\ } \theta\in N\left( \theta_0 \right) \textrm{\ and\ } z\in  B(z^{\prime}).
\end{equation}
In other words, the restriction of the above defined map $F:\partial \mathbb{D}^{2}\times N(z_{0})\times N(z_{0})\rightarrow \mathbb{R}$ on the set $ B(z^{\prime}) \times  N\left( \theta_0 \right) \times \{ z_0\}$ does not attain the value $0\in \mathbb{R}.$ By continuity and compactness, let $M>0$ be the minimum of $F$ on $ B(z^{\prime}) \times  N\left( \theta_0 \right) \times \{ z_0\} .$

Let $ \Theta_K = \left\{ \theta_k \bigm\vert k=1,2,\ldots \right\}$ be an enumeration of the set $\Theta \cap  N\left( \theta_0 \right) $ for which we proved that
\[ d_{\theta_k} (z_n, z_0) >M \textrm{\ for\ all\ } z_n \textrm{\ and\ } \theta_k \in \Theta_K .  \]
Then
\[ d(z_n, z_0) = \sum\limits_{\theta_i \in \Theta} w_i d_{\theta_i}(z_n,z_0) \geq
 \sum\limits_{\theta_k \in \Theta_K} w_k d_{\theta_k}(z_n,z_0)> M \sum_k w_k >0
\]
which contradicts the assumption $d\left( z_0, z_n \right)\rightarrow0 .$
This completes the proof of the proposition.
\end{proof}

\begin{theorem}
\label{basicmetric}The metric space $\left( P,d\right) $ is a
geodesic metric space whose geodesics are precisely the curves $\left\{
c\cap P\bigm\vert c\in \mathcal{C}\right\} .$ In particular, $\left(
P,d\right) $ is uniquely geodesic.
\end{theorem}
\begin{proof}
Let $\left[ x,y\right] $ be the segment with endpoints $x,y\in P.$ In other
words (see terminology introduced in property (C1)), $\left[ x,y\right] $ is
the restriction of $c_{xy}$ to an appropriate interval $I$ such that ${c_{xy}%
}_{\bigm\vert I}:I\rightarrow P$ has endpoints $x$ and $y.$ As all curves in
$\mathcal{C}$ are assumed to be continuous and injective with
respect to the Euclidean topology of $P$ so is the restriction  ${c_{xy}}_{\bigm\vert I}.$ \\
Proposition \ref{equivTOP} implies that
\begin{equation}
{c_{xy}}_{\bigm\vert I}:I\rightarrow P\text{\ is\ continuous\ with\ respect\
to \ the topology\ induced\ by\ }d.  \label{contC}
\end{equation}%
We will next show that
\begin{equation}
d(x,z)+d(z,y)=d(x,y) \text{\ for\ every\ } z\in[x,y].  \label{additive}
\end{equation}
Let $z\in \left[ x,y\right] .$ If $\theta$  is not an endpoint of the curve $
c_{xy},$ then $x_{\theta}\neq y_{\theta}$ and $z$ is contained in
the region bounded by $\left( \theta,x_{\theta} \right) ,\left( \theta,y_{\theta} \right) $ and the
arc $\wideparen{x_{\theta}y_{\theta}}$ (which is the arc not containing $\theta$).
By property (C5) $\left( \theta,z_{\theta} \right)$ does not intersect neither $\left( \theta,x_{\theta} \right)$ nor $\left( \theta,y_{\theta} \right) .$ It follows that $z_{\theta}\in
\wideparen{x_{\theta}y_{\theta}}.$ Moreover, the calculation (\ref%
{triangleequality}) carried out in Case A of Lemma \ref{dTHETA} holds
verbatim, that is,
\begin{equation}
d_{\theta}\left( x,y\right) =d_{\theta}\left( x,z\right) +d_{\theta }\left(
z,y\right) .   \label{sumequality}
\end{equation}
On the other hand, if $\theta$ is an endpoint of the curve $c_{xy},$ then $%
x_{\theta}=z_{\theta}=y_{\theta}$ and the above inequality holds trivially.
It follows that equality (\ref{sumequality}) holds for all $\theta \in \Theta
$ and hence $d\left( x,y\right) =d\left(x,z\right) +d\left( z,y\right) .$

The additive property (\ref{additive}) holds for any three points in $[x,y]$
which implies that
\begin{equation}
{c_{xy}}_{\bigm\vert I} :I\rightarrow P \mathrm{\ has\ finite\ length.}
\label{flength}
\end{equation}
The latter property along with (\ref{contC}) assert that ${c_{xy}}_{%
\bigm\vert I}$ can be parametrized by arclength. It is well known (see, for
example, Proposition 2.2.7 in \cite{Pap}) that a curve with arclength
parametrization and endpoints $x,y$ is a geodesic segment with respect to a
metric $d$ if and only if for every $z$ in the curve we have $%
d(x,z)+d(z,y)=d(x,y).$ This completes the proof that $\left[ x,y\right] $ is
a geodesic with respect to $d.$

Last we show that the segment $\left[ x,y\right] $ is the unique geodesic
segment with respect to $d$ joining $x,y.$ Assume, on the contrary, that
there exists a geodesic joining $x,y$ which contains a point $z\notin \left[
x,y\right] .$ By the previous discussion, we may assume that $y\notin \left[
x,z\right] $ and $x\notin \left[ y,z\right] .$  In other words, $z$ is not a
point of the curve $c_{xy} .$

Let $\xi_{xy}$ be one endpoint of $c_{xy} .$ Clearly,
\begin{equation*}
d_{\xi_{xy}}\left( x,y\right) =0\lvertneqq d_{\xi_{xy}}\left( x,z\right) .
\end{equation*}
By the triangle inequality for the pseudo-metric $d_{\xi_{xy}}$ we have
\begin{equation*}
d_{\xi_{xy}}\left( x,y\right) \lvertneqq d_{\xi_{xy}}\left(x,z\right)
+d_{\xi_{xy}} \left( z,y\right)
\end{equation*}
and
\begin{equation*}
d_{\theta}\left( x,y\right) \leq d_{\theta}\left( x,z\right)
+d_{\theta}\left( z,y\right) \mathrm{\ for\ \ all\ }\theta \neq \xi_{xy}.
\end{equation*}
It follows that $d\left( x,y\right) \lvertneqq d\left( x,z\right) +d\left(
z,y\right) $ which contradicts the fact that $z$ lies on a geodesic joining $%
x,y.$ \newline
\end{proof}

\begin{lemma}
\label{rayINF} Let $x\in P,$ $v$ a point in the boundary of $P$ and $\left[
x,v\right] $ the unique segment obtained from $\mathcal{C}.$ By the above
Theorem, $\left[ x,v\right) =\left[ x,v\right] \setminus\left\{ v\right\} $
can be viewed as a geodesic (with respect to $d$) ray $r_{v}$ of $P.$
Then, for any sequence $\left\{ y_{n}\right\} \subset \lbrack x,v)$
converging to $v$ in the Euclidean sense we have
\begin{equation*}
d\left( x,y_{n}\right) \rightarrow \infty \mathrm{\ as\ \ }n\rightarrow
\infty .
\end{equation*}%
In particular, the geodesic ray $r_{v}$ is realized by an isometry $r_{v}:
\left[ 0,\infty \right) \longrightarrow P$ and $\left( P,d\right) $ a
geodesically complete metric space.
\end{lemma}
\begin{proof}
Let $\theta_{0}, \xi_{\theta_{0}}$ be the endpoints of the curve $c$ which
is the unique curve determined by the segment of $\partial P$ containing $v.$
By assumption in Definition \ref{defd}, $\theta_{0}\in \Theta$ and it
suffices to show that
\begin{equation*}
d_{\theta_{0}}(x,y_{n})\rightarrow \infty \mathrm{\ as\ \ }n\rightarrow
\infty.
\end{equation*}
We claim that $\left( y_{n}\right) _{\theta _{0}}\rightarrow \xi_{\theta_{0}}
$ or, equivalently, $\left\vert
\wideparen{ \xi_{\theta_{0}}
\left( y_{n}\right) _{\theta _{0}}}\right\vert \longrightarrow 0.$ To see
this, assume, on the contrary, that
\begin{equation}
\exists \varepsilon _{0}>0:\forall n\in \mathbb{N},\exists N>n\mathrm{\ \
with\ \ } \left\vert \wideparen{ \xi_{\theta_{0}} \left( y_{N}\right)
_{\theta _{0}}}\right\vert \geq \varepsilon _{0}.  \label{thetaDelta}
\end{equation}%
Pick a point $\theta _{\delta }\in
\wideparen{  \xi_{\theta_{0}} \left(
y_{1}\right) _{\theta _{0}}}$ such that $\left\vert \wideparen{
\xi_{\theta_{0}} \theta _{\delta }}\right\vert <\varepsilon _{0}.$ The line $
\left( \theta_{0} ,\theta _{\delta }\right) $ must intersect $\left[ x,v%
\right] $ at a point, say, $y_{\delta }.$ For all $n\in \mathbb{N}$ such
that $y_{n}\in \lbrack y_{\delta },v)$ we have $\left\vert \wideparen{
\xi_{\theta_{0}} \left( y_{n}\right) _{\theta _{0}}}\right\vert < \left\vert
\wideparen{ \xi_{\theta_{0}} \theta _{\delta }}\right\vert <\varepsilon _{0}$
contradicting (\ref{thetaDelta}).

This shows that $\left \vert
\wideparen{
 \xi_{\theta_{0}}
\left(  y_{n}\right)  _{\theta_{0}}}\right \vert \longrightarrow0$ or,
equivalently, $\log \frac{1}{\left \vert \wideparen{\xi_{\theta_{0}}\left(
y_{n}\right)  _{\theta_{0}}}\right \vert }\rightarrow \infty.$ As $\left
\vert \wideparen{\eta_{\theta_{0}}\left( y_{n}\right)  _{\theta_{0}}}\right
\vert < \left \vert \wideparen{\eta_{\theta_{0}}\xi_{\theta_{0}}}\right \vert
$ for all $n$ it follows
\begin{align}
d_{\theta_{0}}(x,y_{n})= & \log \frac { \left \vert \wideparen{\eta_{%
\theta_{0}}\left(y_{n}\right) _{\theta_{0}}}\right \vert \left \vert %
\wideparen{\xi_{\theta_{0} }x_{\theta_{0}}}\right \vert } { \left \vert %
\wideparen{\eta_{\theta_{0}}x_{\theta_{0}}}\right \vert \left \vert %
\wideparen{\xi_{\theta_{0}}\left( y_{n}\right) _{\theta_{0}}}\right \vert }
\leq \log \frac { \left \vert \wideparen{\eta_{\theta_{0}}\xi_{\theta_{0}}}%
\right \vert \left \vert \wideparen {\xi_{\theta_{0}}x_{\theta_{0}}}\right
\vert } { \left \vert \wideparen{\eta _{\theta_{0}}x_{\theta_{0}}}\right
\vert \left \vert \wideparen{\xi_{\theta_{0}}\left( y_{n}\right)
_{\theta_{0}}}\right \vert } =  \notag \\
& = \log \frac{\left \vert \wideparen{\eta_{\theta_{0}}\xi_{\theta_{0}}}%
\right \vert \left \vert \wideparen {\xi_{\theta_{0}}x_{\theta_{0}}}\right
\vert }{\left \vert \wideparen{\eta _{\theta_{0}}x_{\theta_{0}}}\right \vert
}+\log \frac{1}{\left \vert \wideparen {\xi_{\theta_{0}}\left( y_{n}\right)
_{\theta_{0}}}\right \vert }\rightarrow \infty  \notag
\end{align}
as required.
\end{proof}

In a similar manner the following can be seen: let $v,w$ be two points in
the boundary of $P$ contained in distinct segments of $\partial P.$ Then $%
(v,w) := c_{v,w} \cap P $ is a geodesic line in $P$ of infinite length.


\section{Generalizations}

In this Section let $\left( X,\rho \right) $ be a proper geodesic metric space
homeomorphic to $\mathbb{R}^{2}$ with the following uniqueness property:

\begin{enumerate}
\item[(G1)] for any two points $x,y\in X$ there exists a unique geodesic
segment $\sigma _{xy}:I\rightarrow X,$ where $I$ is an interval in $\mathbb{R},$ with endpoints $x$ and $y.$ We denote $\sigma _{xy}$ by $[x,y].$ Moreover,  every segment $\left[ x,y\right] $ extends uniquely to a geodesic line (isometry) $\sigma :\mathbb{R}\rightarrow X.$ We will be
saying that the line $\sigma $ contains the segment $\left[ x,y\right] .$
\end{enumerate}

Denote by $\partial X$ the boundary at infinity of $X$ defined via asymptotic geodesic rays (see, for example, \cite[page 260]{BrH}). We assume that $\partial X ,$ equipped with the topology of uniform convergence on compact sets, is homeomorphic to $\mathbb{S}^1$ and $\partial X$ compactifies $X$ so that $X\cup \partial X$ is homeomorphic to the closed unit disk.

We further assume that $X\cup \partial X$ satisfies the following properties
\begin{enumerate}
 \item[(G2)]  for every two points $\xi , \eta \in \partial X$ there exists a unique geodesic line in $X$ joining them.
  \item[(G3)] for every two points $x\in X$ and $\xi \in \partial X$ there exists a unique geodesic ray in $X$ joining them.
\end{enumerate}
The class of such geodesic metric spaces satisfying properties (G1), (G2) and (G3) includes universal coverings of closed surfaces of genus $\geq 2$ with a Riemannian metric of non-positive curvature.

We will say that $P$ is a convex subset of $X$ if for any two points $x,y\in
P$ the segment $\left[ x,y\right] $ is entirely contained in $P.$ We say
that $P$ is a convex polygon in $X$ if $P$ is an open bounded convex subset
of $X$ whose boundary is a finite union of geodesic segments.

As properties (G1), (G2) and (G3)  above are identical with properties (C1), (C2) and (C3) given  at the beginning of Section 2, the construction carried out in Section 2 can be applied verbatim with the class $\mathcal{C}$ being the geodesic lines in $X. $ This defines a new metric $d$ on the convex polygon $P$ such that $(P,d)$ is a geodesically complete  metric space whose geodesic lines are precisely the curves $\left\{c\cap P\bigm\vert c\textrm{\ geodesic\ line\ in\ }X\right\} .$

However, a proper geodesic metric space homeomorphic to $\mathbb{R}^{2}$ need not have a boundary satisfying property (C2). The Euclidean space $\mathbb{R}^{2}$ itself provides such an example. For this class of metric spaces we carry out in the next Subsection a  construction analogous to the one given in Section 2 by deploying a convex polygon containing the given polygon $P.$ In the special case of $\mathbb{R}^2$ we describe, in Subsection \ref{r2c}  below, an analogous procedure for putting a metric on a polygon $P$ which does not depend on the choice of a convex polygon containing $P.$

\subsection{Generalization to geodesic metric spaces}
Recall that $\left( X,\rho \right) $ denotes  a proper geodesic metric space
homeomorphic to $\mathbb{R}^{2}$ satisfying properties (G1)-(G3).
\begin{lemma}
Given any convex polygon $P$ in $X$ there exists a convex polygon $K$ in $X$ with $P\cup \partial P\subset K .$
\end{lemma}
\begin{proof}
 Denote by $A_{1},A_{2},\ldots A_{n}$ the vertices of $P.$ For each side $\left[
A_{i},A_{i+1}\right] $ ($i=1,\ldots ,n$ with $A_{n+1}\equiv A_{1}$) consider
the geodesic (with respect to the geometry of $X$) line $\sigma _{i}$
containing the side $\left[ A_{i},A_{i+1}\right] .$ \newline
Claim: for all $i,$ $\sigma _{i}\cap P=\varnothing .$

To check this assume $\sigma _{i}$ intersects $P.$ Then $P\setminus\sigma
_{i}$ consists of two or more components and denote by $P_{1}$ the component
whose boundary contains $\left[ A_{i},A_{i+1}\right] .$ If $P_{2}$ is an
other component of $P\setminus\sigma _{i}$ then at least one of the vertices 
$A_{i},A_{i+1}$ is not contained in $\partial P_{2}.$ Say, $A_{i}\notin
\partial P_{2}$ and pick any $x\in P_{2}.$ Then the geodesic segment $\left[
A_{i},x\right] $ must

\begin{enumerate}
\item[either,] contain $\left[ A_{i},A_{i+1}\right] $ which violates (G2)
because $x\notin \sigma _{i}$

\item[or, ] intersect $\sigma _{i}$ at a point $y\in \sigma _{i}$ which
violates (G1).
\end{enumerate}

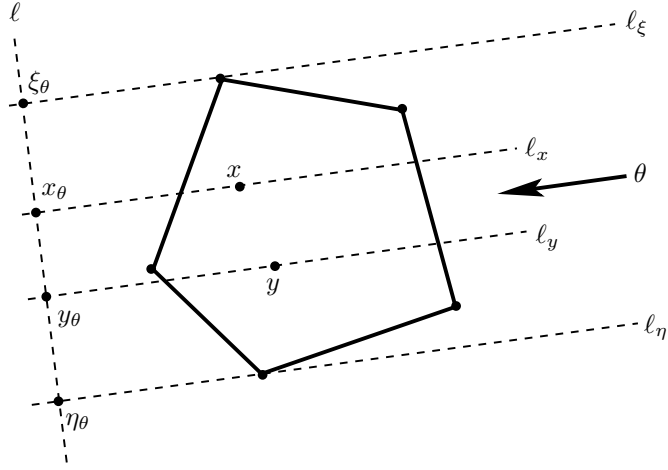
\begin{figure}[ptb]
\begin{center}
\includegraphics[scale=1.55]{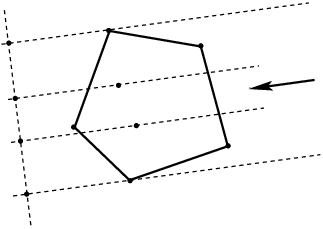}
\begin{picture}(0,0)
\put(-225,56){$y_{\theta}$}
\put(-146,68){$y$}
\put(-3,52){$ \ell_{\eta}$}
\put(-10,167){$ \ell_{\xi}$}
\put(-243,170){$ \ell$}
\put(-44,86){$\ell_y $}\put(-48,119){$\ell_x$}
\put(-161,111){$x$}
\put(-231,104){$x_{\theta}$}
\put(-7,109){$\theta$}
\put(-222,17){$\eta_{\theta}$}
\put(-236,144){$ \xi_{\theta}$}
\end{picture}\\[7mm]
\end{center}
\par
\label{figlast}
\caption{The strip which bounds $P$ in the $\theta$ direction and the projection points whose distances define $d_{\theta}$}
\end{figure}

Notation: for a geodesic line $h$ not intersecting $P$ we will be writing $%
h^{+}$ for the component of $X\setminus h$ which contains $P.$

We next construct a convex polygon $K$ containing $P.$ For each of the $n$ sides  $\left[ A_{i},A_{i+1}\right] $
of the polygon $P$ consider a geodesic (with respect to the geometry of $X$) line $\ell _{i}$
as follows:
choose points $B_{i} \in \sigma_{i-1}$ and $C_i \in \sigma_{i+1}$ such that
 \[ \left[ A_{i-1},B_{i}\right] =\left[ A_{i-1},A_{i}\right] \cup \left[ A_{i},B_{i}\right]
 \mathrm{\ and\ }
\left[ C_i , A_{i+2}\right] =\left[ C_i ,A_{i+1}\right] \cup \left[ A_{i+1},A_{i+2}\right]
 \]
Set $\ell _{i}$ to be the unique geodesic line containing $B_{i},C_{i} .$
Observe that for all $i,$ $\ell _{i}\cap P=\varnothing .$ To
check this note that if $\ell_i$ intersects $P$ then $\ell_i $ must intersect\\[3mm]
\hspace*{1cm}
\begin{minipage}{13cm}
\begin{enumerate}
 \item[either,]$\sigma _{i-1}$ at a point other than $B_{i}\in \sigma _{i-1}\cap \ell_{i}$
 \item[or,] $\sigma _{i+1}$ at a point other than $C_{i}\in \sigma_{i+1}\cap \ell _{i}.$
\end{enumerate}
\end{minipage}\\[3mm]
 In both cases we have, by (G1), a contradiction.

Set $\displaystyle K^{\prime }=\cap _{i=1}^{n}\ell _{i}^{+}.$ Clearly, $%
K^{\prime }$ is convex and contains $P.$ If $K^{\prime }$ is bounded we set
the desired bounded convex polygon $K$ to be $K^{\prime }.$ If $K^{\prime }$
is not bounded then each end $K_{j}^{\prime },j=1,\ldots ,k$ of $K^{\prime }$
is determined by two subrays $r_{j}$ and $r_{j+1}$ of $\ell _{j}$ and $\ell
_{j+1}$ respectively. Pick a line $\ell _{K_{j}^{\prime }}$ intersecting
both $r_{j}$ and $r_{j+1}$ and then the intersection 
\begin{equation*}
\displaystyle K=K^{\prime }\cap \left( \displaystyle\cap _{j=1}^{k}\ell
_{K_{j}^{\prime }}^{+}\right)
\end{equation*}%
is the desired bounded convex polygon containing $P.$ \end{proof}

Given any convex bounded polygon $P$ in $X$ consider and fix a convex
polygon $K$ containing $P$ in its interior. As explained above, such a
polygon always exists. The boundary $\partial K$ is homeomorphic to the
circle. We may now perform the construction described in Section \ref%
{sectionPrescribed} where the geodesic segments in the metric space $\left(
X,\rho \right) $ with endpoints on $\partial K$ constitute a collection of
curves satisfying properties (C1)-(C4).

The Euclidean length $\left \vert \wideparen{\theta_{i}\theta_{j}}\right \vert
$ used to define the pseudo-metric $d_{\theta}$ in $P$ is the only
adjustment needed: for points  $\theta _{1},\theta _{2},\theta _{3},\theta
_{4}$ in cyclic clockwise order in $\partial K$,  denote by $\wideparen{%
\theta_{i}\theta_{j}}$ the piece-wise geodesic curve in $\partial K$ with
endpoints  $\theta_{i},\theta_{j}$ and by $\left \vert \wideparen{%
\theta_{i}\theta_{j}}\right \vert $ its length with respect to the metric  $%
\rho$ of $X.$ Then the pseudo-metric $d_{\theta}$ is given by (\ref{starTHETA})
in an identical way.

The metric space $\left( P,d\right) $ obtained in this way is a geodesic
metric space whose geodesics are precisely the curves 
\begin{equation*}
\left\{ \sigma \cap P\bigm\vert\sigma \mathrm{\ is\ a\ geodesic\ line\ in\ }%
\left( X,\rho \right) \right\} .
\end{equation*}%
In particular, $\left( P,d\right) $ is uniquely geodesic (G1) and
geodesically complete (G2).

\subsection{The special case $\mathbb{R}^2$\label{r2c}}
Let $P$ be a convex polygon in Euclidean space $\mathbb{R}^2$. We may view $\mathbb{S}^1 \subset \mathbb{R}^2$ as the set of directions in $\mathbb{R}^2$ where each $x\in \mathbb{S}^1$ determines a unique angle $\theta \in \left[0,2\pi \right).$

For each direction $\theta$ there exist exactly two parallel lines, say $\ell_{\xi} , \ell_{\eta}$ such that the strip bounded by them contains  $P$ and the strip is minimal with respect to this property. Pick any line $\ell$ perpendicular to $\ell_{\xi} , \ell_{\eta}$ and denote by $\ell_x$ (resp. $\ell_y$) the line which contains $x$ (resp. $y$) and is parallel to $\ell_{\xi}$ (see Figure 5).
Set

$\xi_{\theta}$ the intersection point $\ell_{\xi}\cap \ell$

$\eta_{\theta}$ the intersection point $\ell_{\eta}\cap \ell$

$x_{\theta}$ the intersection point $\ell_{x}\cap \ell$

$y_{\theta}$ the intersection point $\ell_{y}\cap \ell$

\noindent We may define  $d_{\theta }$ as in Section 2, equation (\ref{starTHETA})
\begin{equation*}
d_{\theta }\left( x,y\right) :=\log \left[ \eta _{\theta },y_{\theta
},x_{\theta },\xi _{\theta }\right]
=
\log \frac{\left \vert\eta_{\theta}x_{\theta}\right \vert
\left \vert \xi_{\theta}y_{\theta}\right \vert }
{\left \vert \eta_{\theta}y_{\theta}\right \vert \left \vert \xi_{\theta}x_{\theta}\right \vert }
\end{equation*}
where $\left \vert\ \,\,\,\,\right \vert$ stands for Euclidean distance. As in Lemma \ref{dTHETA}, it can be seen that $d_{\theta}$ is a pseudo-distance on $P.$

 Let  $\Theta$  be a dense subset of $\mathbb{S}^1$ containing all
directions determined by the sides of $P.$  As in Definition \ref{defd}, to each $\theta _{i}$ in $\Theta $ assign a positive real $w_{i}$ such that the series $\sum\limits_{i}w_{i}$ converges. For $x,y$ $\in P$ define
\begin{equation*}
d\left( x,y\right) :=\sum\limits_{i=1}^{\infty }w_{i}\,d_{\theta _{i}}\left(
x,y\right) .
\end{equation*}
Working in an identical way as in Section 2 it can be seen that $d$ is a metric  making $P$ a geodesically complete metric space whose geodesic lines are precisely the (open) Euclidean segments in $P.$

\section{Further properties of $\left( P,d\right) $}

\noindent\ For a convex domain $U$ in $\mathbb{R}^{2}$ denote by $d_{%
\mathcal{H}}$ the Hilbert metric on $U$ for which we refer the reader to 
\cite[Ch.5, Section 6]{Pap}. Let $\left( P,d\right) $ be the metric space
obtained in Section \ref{sectionPrescribed} with the collection of curves $%
\mathcal{C}$ being the Euclidean lines.

\begin{theorem}
\label{nonisometric} For any convex domain $U$ in $\mathbb{R}^{2}$ equipped
with the Hilbert metric $d_{\mathcal{H}}$ the metric spaces $\left(
P,d\right) $ and $\left( U,d_{\mathcal{H}}\right) $ are not isometric.
\end{theorem}

\begin{proof}
Before dealing with the proof of the Theorem, we will state and prove a claim concerning
geodesic rays in $P$ which have endpoints in the same side of $P.$

Let $\left[ u,v\right] $ be a side of $P,$ where $u,v$ are vertices in $%
\partial P$, and consider two parallel (in the Euclidean sense) geodesic
rays $\left[ x,x_{\infty }\right] $ and $\left[ y,y_{\infty }\right] $ with $%
x_{\infty },y_{\infty }\in \left[ u,v\right] $ and $y_{\infty }\in \left[
x_{\infty },v\right] $ Pick sequences $\left\{ x_{n}\right\} _{n\in \mathbb{N%
}}\subset \left[ x,x_{\infty }\right] $ and $\left\{ y_{n}\right\} _{n\in 
\mathbb{N}}\subset \left[ y,y_{\infty }\right] $ such that for all $n\in 
\mathbb{N}$ the segment $\left[ x_{n},y_{n}\right] $ is parallel to $\left[
u,v\right] $ and, in addition, $x_{n}\rightarrow x_{\infty }$ and $%
y_{n}\rightarrow y_{\infty }.$ In other words, 
\begin{equation*}
\lim_{n\rightarrow \infty}d\left( x,x_{n}\right) =\infty =\lim_{n\rightarrow \infty}d\left( y,y_{n}\right) .
\end{equation*}%
\begin{figure}[tbp]
\begin{center}
\includegraphics[scale=1.30]{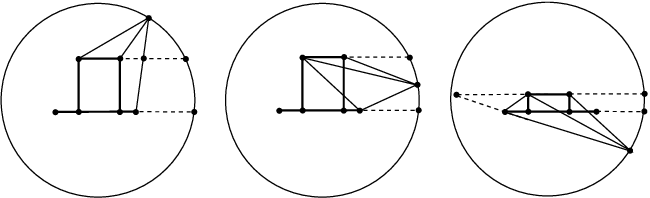} \vspace*{-4mm}

\hspace*{2.1cm}
\begin{picture}(5,0)
\put(-157,61){$\scriptstyle x_{\infty}$}
\put(-158,105){$\scriptstyle x_{\lambda}$}
\put(-133,105){$\scriptstyle y_{\lambda}$}
\put(-131,61){$\scriptstyle  y_{\infty}$}
\put(-82,102){$\scriptstyle  \theta_{xy} $}
\put(-76,65){$\scriptstyle \theta_{\infty}$}
\put(-107,129){$\scriptstyle \theta_{\lambda}$}
\put(-113,62){$\scriptstyle v$}
\put(-172,62){$\scriptstyle u$}
\put(-108,105){$\scriptstyle z_{\lambda}$}
\put(-17,61){$\scriptstyle x_{\infty}$}
\put(-18,106){$\scriptstyle x_{\lambda}$}
\put(9,107){$\scriptstyle y_{\lambda}$}
\put(9,61){$\scriptstyle  y_{\infty}$}
\put(58,102){$\scriptstyle  \theta_{xy} $}
\put(64,84){$\scriptstyle \theta_{\lambda}$}
\put(64,65){$\scriptstyle \theta_{\infty}$}
\put(27,62){$\scriptstyle v$}
\put(-32,62){$\scriptstyle u$}
\put(131,63){$\scriptstyle x_{\infty}$}
\put(122,84){$\scriptstyle x_{\lambda}$}
\put(149,84){$\scriptstyle y_{\lambda}$}
\put(84,83){$\scriptstyle z_{\lambda}$}
\put(149,61){$\scriptstyle  y_{\infty}$}
\put(206,79){$\scriptstyle  \theta_{xy} $}
\put(198,40){$\scriptstyle \theta_{\lambda}$}
\put(206,65){$\scriptstyle \theta_{\infty}$}
\put(175,62){$\scriptstyle v$}
\put(108,62){$\scriptstyle u$}
\put(-157,1){Case I}\put(-17,1){Case II}\put(122,1){Case III}
\end{picture}\\[7mm]
\end{center}
\par
\label{fig2}
\caption{The three cases considered in the Claim in the proof of Theorem 
\protect\ref{nonisometric}.}
\end{figure}
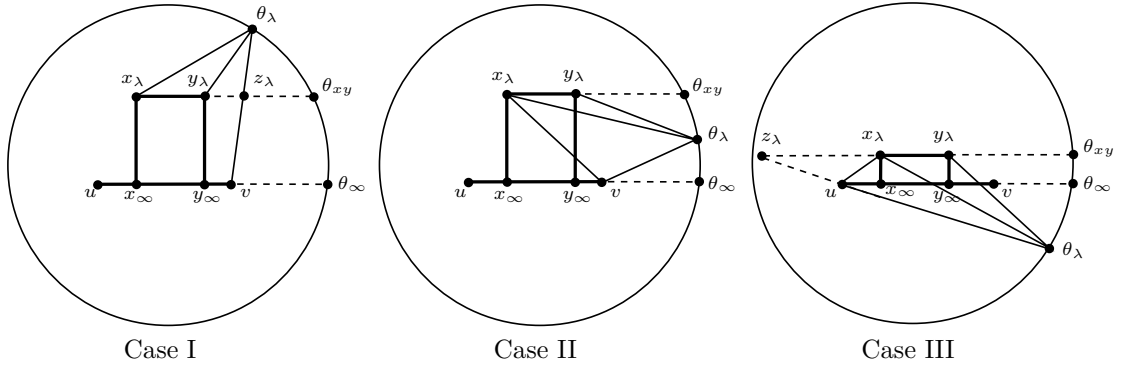

\noindent \textbf{Claim:} The set $\left\{ d\left( x_{n},y_{n}\right) 
\bigm
\vert n\in \mathbb{N}\right\} $ is bounded.\newline
Proof of Claim. As $d\left( x_{n},y_{n}\right) =\sum\limits_{k=1}^{\infty
}w_{k}\,d_{\theta _{k}}\left( x_{n},y_{n}\right) $ it suffices to show that 
\begin{equation*}
\left\{ d_{\theta _{k}}\left( x_{n},y_{n}\right) \bigm \vert k,n\in \mathbb{N%
}\right\} 
\end{equation*}%
is bounded. Assume that it is not. Then there must exist a sequence 
\begin{equation}
\left\{ d_{\theta _{k(\lambda )}}\left( x_{n(\lambda )},y_{n(\lambda
)}\right) \right\} _{\lambda =1}^{\infty }  \label{sequence}
\end{equation}%
converging to $\infty $ as $\lambda \rightarrow \infty .$ To simplify
notation we write $d_{\theta _{\lambda }}\left( x_{\lambda },y_{\lambda
}\right) $ instead of $d_{\theta _{k(\lambda )}}\left( x_{n(\lambda
)},y_{n(\lambda )}\right) .$ The Euclidean line extending the side $\left[
u,v\right] $ intersects the boundary of the unit disk in two points denoted
by $\theta _{\infty }$ and $\theta _{-\infty }.$ Clearly, $\left\{ d_{\theta
_{\lambda }}\left( x_{\lambda },y_{\lambda }\right) \right\} $ is bounded
for all $\theta _{\lambda }$ away from $\theta _{\infty }$ and $\theta
_{-\infty }.$ We will examine $\left\{ d_{\theta _{\lambda }}\left(
x_{\lambda },y_{\lambda }\right) \right\} $ for $\theta _{\lambda }$ close
to $\theta _{\infty }$ and an identical argument will work for $\theta
_{-\infty }.$ We may assume that the segment $\left[ u,\theta _{\infty }%
\right] $ contains $v.$ 
\begin{figure}[tbp]
\begin{center}
\includegraphics[scale=2.1]{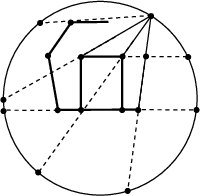} 
\begin{picture}(0,0)
\put(-126,76){$ x_{\infty}$}
\put(-136,145){$x_{\lambda}$}
\put(-94,145){$y_{\lambda}$}
\put(-56,145){$z_{\lambda}$}
\put(-88,76){$ y_{\infty}$}
\put(-12,138){$\theta_{xy} $}
\put(-3,82){$\theta_{\infty}$}
\put(-52,186){$\theta_{\lambda}$}
\put(-64,77){$ v$}
\put(-157,77){$ u$}
\put(-84,-6){$\eta_{\theta_{\lambda}}$}
\put(-228,81){$\theta_{-\infty}$}
\put(-238,99){$\left(x_{\lambda}\right)_{\theta_{\lambda}}$}
\put(-186,8){$\left(y_{\lambda}\right)_{\theta_{\lambda}}$}
\put(-177,183){$\xi_{\theta_{\lambda}}$}
\end{picture}\\[7mm]
\end{center}
\par
\label{fig3}
\caption{Relevant notation in Case I of the Claim in the proof of Theorem 
\protect\ref{nonisometric}.}
\end{figure}
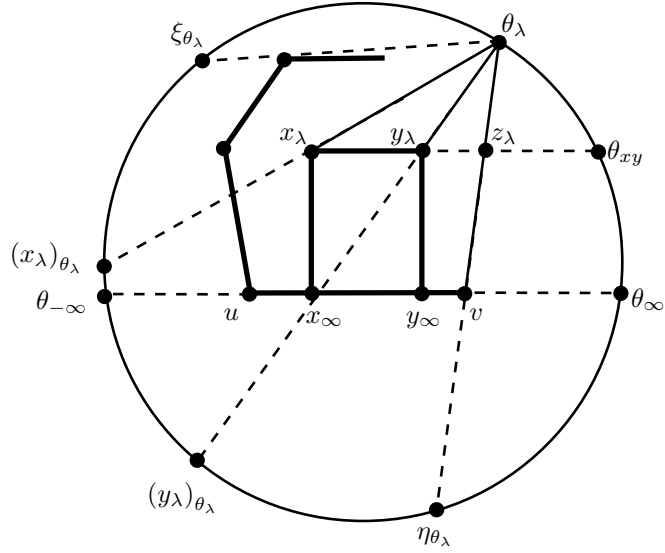
\newline
We distinguish 3 cases as depicted in Figure 6.
Denote by $\theta _{xy}$ the intersection of $\partial \mathbb{D}^{2}$
with the extension of $\left[ x_{\lambda },y_{\lambda }\right] $ such that $%
y_{\lambda }\in \left[ x_{\lambda },\theta _{xy}\right] .$\newline
Case I: $\theta _{xy}\in \wideparen{\theta _{\lambda }\theta _{\infty }}.$%
\newline
Case II: $\theta _{\lambda }\in \wideparen{\theta _{xy}\theta _{\infty }}.$%
\newline
Case III: $\theta _{\infty }\in \wideparen{\theta _{\lambda }\theta _{xy}}.$%
\newline
We discuss in detail Case I (see Figure 7).
The distance $d_{\theta
_{\lambda }}\left( x_{\lambda },y_{\lambda }\right) $ is by definition (see
equations (\ref{starTHETAA}) and (\ref{starTHETA})) the sum of the
logarithms of two ratios 
\begin{equation*}
d_{\theta _{\lambda }}\left( x_{\lambda },y_{\lambda }\right) =\log \frac{%
\left\vert \wideparen{\eta _{\theta _{\lambda }}\left( x_{\lambda }\right)
_{\theta _{\lambda }}}\right\vert }{\left\vert \wideparen{\eta _{\theta
_{\lambda }}\left( y_{\lambda }\right) _{\theta _{\lambda }}}\right\vert }%
+\log \frac{\left\vert \wideparen{\xi _{\theta _{\lambda }}\left( y_{\lambda
}\right) _{\theta _{\lambda }}}\right\vert }{\left\vert \wideparen{\xi
_{\theta _{\lambda }}\left( x_{\lambda }\right) _{\theta _{\lambda }}}%
\right\vert }
\end{equation*}%
For large enough $\lambda $ both projections $\left( x_{\lambda }\right)
_{\theta _{\lambda }}$ and $\left( y_{\lambda }\right) _{\theta _{\lambda }}$
approach $\theta _{-\infty }$ and therefore the second summand is bounded for
all large enough $\lambda .$ We will reach a contradiction by showing that
the first summand or, equivalently, the ratio $x_{\lambda }\widehat{\theta
_{\lambda }}z_{\lambda }/z_{\lambda }\widehat{\theta _{\lambda }}y_{\lambda }
$ of the corresponding angles is bounded, where $z_{\lambda }$ is the point
where the extension of $\left[ x_{\lambda },y_{\lambda }\right] $ meets $%
\left[ \theta _{\lambda },v\right] .$  As $\sin a<a<2\sin a$ for small
enough $a$ we may work with the ratio 
\begin{equation*}
\frac{\sin \left( x_{\lambda }\widehat{\theta _{\lambda }}z_{\lambda
}\right) }{\sin \left( z_{\lambda }\widehat{\theta _{\lambda }}y_{\lambda
}\right) }
\end{equation*}%
Using the law of sines for the triangles $x_{\lambda }\theta _{\lambda
}z_{\lambda }$ and $y_{\lambda }\theta _{\lambda }z_{\lambda },$ we obtain 
\begin{equation*}
\frac{\sin \left( x_{\lambda }\widehat{\theta _{\lambda }}z_{\lambda
}\right) }{\sin \left( z_{\lambda }\widehat{\theta _{\lambda }}y_{\lambda
}\right) }=\frac{\left\vert x_{\lambda }z_{\lambda }\right\vert
\,\,\left\vert y_{\lambda }\theta _{\lambda }\right\vert }{\left\vert
y_{\lambda }z_{\lambda }\right\vert \,\,\left\vert x_{\lambda }\theta
_{\lambda }\right\vert }.
\end{equation*}

Clearly, the right hand side of the above equality converges to $\frac {%
\left \vert x_{\infty}v\right \vert \, \, \left \vert
y_{\infty}\theta_{\infty }\right \vert }{\left \vert y_{\infty}v\right \vert
\, \, \left \vert x_{\infty }\theta_{\infty}\right \vert }$ as $\lambda
\rightarrow \infty,$ which is a positive real number depending on the
Euclidean distances of the boundary points $x_{\infty},y_{\infty},v$ and $%
\theta_{\infty}.$ Therefore we have shown that the set
\begin{equation*}
\left \{ d_{\lambda}\left( x_{\lambda},y_{\lambda}\right) \bigm \vert %
\lambda \mathrm{\ satisfies\ Case\ I}\right \}
\end{equation*}
is bounded.

In Case II, adopting the same reasoning, we will reach a contradiction by
showing that the ratio $y_{\lambda}\widehat{\theta_{\lambda}}v$ /$x_{\lambda
}\widehat{\theta_{\lambda}}v$ is bounded. As 
\begin{equation*}
\frac{y_{\lambda}\widehat{\theta_{\lambda}}v}{x_{\lambda}\widehat {%
\theta_{\lambda}}v}=\frac{x_{\lambda}\widehat{\theta_{\lambda}}v+x_{\lambda }%
\widehat{\theta_{\lambda}}y_{\lambda}}{x_{\lambda}\widehat{\theta_{\lambda}}v%
}=1+\frac{x_{\lambda}\widehat{\theta_{\lambda}}y_{\lambda}}{x_{\lambda }%
\widehat{\theta_{\lambda}}v}
\end{equation*}
it suffices to bound the ratio 
\begin{equation*}
\frac{\sin \left( x_{\lambda}\widehat{\theta_{\lambda}}y_{\lambda}\right) }{%
\sin \left( x_{\lambda}\widehat{\theta_{\lambda}}v\right) }.
\end{equation*}
Using again the law of sines for the triangles $x_{\lambda}\theta_{\lambda
}y_{\lambda}$ and $x_{\lambda}\theta_{\lambda}v$ we obtain 
\begin{align}
\frac{\sin \left( x_{\lambda}\widehat{\theta_{\lambda}}y_{\lambda}\right) }{%
\sin \left( x_{\lambda}\widehat{\theta_{\lambda}}v\right) }&=\frac {\sin
\left( y_{\lambda}\widehat{x_{\lambda}}\theta_{\lambda}\right) }{\sin \left(
v\widehat{x_{\lambda}}\theta_{\lambda}\right) }\frac{\left \vert
x_{\lambda}y_{\lambda}\right \vert \, \, \left \vert \theta_{\lambda}v\right
\vert }{\left \vert y_{\lambda}\theta_{\lambda}\right \vert \, \, \left
\vert x_{\lambda }v\right \vert }  \notag \\
& < \frac{\sin \left( x_{\lambda}\widehat{\theta_{\lambda}}v\right) }{\sin
\left( v\widehat{x_{\lambda}}\theta_{\lambda}\right) }\frac{\left \vert
x_{\lambda}y_{\lambda}\right \vert \, \, \left \vert \theta_{\lambda}v\right
\vert }{\left \vert y_{\lambda}\theta_{\lambda }\right \vert \, \, \left
\vert x_{\lambda}v\right \vert }\overset{(\ast)}{=}\frac{\left \vert
x_{\lambda}v\right \vert }{\left \vert \theta_{\lambda }v\right \vert }\frac{%
\left \vert x_{\lambda}y_{\lambda}\right \vert \, \, \left \vert
\theta_{\lambda}v\right \vert }{\left \vert
y_{\lambda}\theta_{\lambda}\right \vert \, \, \left \vert x_{\lambda}v\right
\vert }=\frac{\left \vert x_{\lambda}y_{\lambda}\right \vert \, \,}{\left
\vert y_{\lambda}\theta_{\lambda}\right \vert \, \,}  \label{longsinelaw}
\end{align}
where the equality $(\ast)$ follows from the law of sines for the triangle $%
x_{\lambda}\theta_{\lambda}v$ and the inequality follows from the fact the
angle $x_{\lambda}\widehat{\theta_{\lambda}}v$ is always (in Case II)
strictly larger than $y_{\lambda}\widehat{x_{\lambda}}\theta_{\lambda} .$ 

The right hand side of (\ref{longsinelaw}) clearly converges to $\frac{%
\left
\vert x_{\infty }y_{\infty}\right \vert \,}{\left \vert
y_{\infty}\theta_{\infty}\right \vert }$ as $\lambda \rightarrow \infty,$
and as before, it follows that the set 
\begin{equation*}
\left \{ d_{\lambda}\left( x_{\lambda},y_{\lambda}\right) \bigm \vert %
\lambda \mathrm{\ satisfies\ Case\ II}\right \}
\end{equation*}
is bounded.

In Case III observe that the extension of the segment $\left[ x_{\lambda
},y_{\lambda}\right] $ intersects $\left[ \theta_{\lambda},u\right] $ at a
point $z_{\lambda}$ which, for sufficiently large $\lambda,$ lies inside the
unit disk. As in Case I, we use the law of sines for the triangles $%
y_{\lambda}\theta_{\lambda}z_{\lambda}$ and $x_{\lambda}\theta_{\lambda}z_{%
\lambda}$ to obtain 
\begin{equation*}
\frac{\sin \left( y_{\lambda}\widehat{\theta_{\lambda}}z_{\lambda}\right) }{%
\sin \left( x_{\lambda}\widehat{\theta_{\lambda}}z_{\lambda}\right) }=\frac{%
\left \vert y_{\lambda }z_{\lambda}\right \vert \, \, \left \vert
x_{\lambda}\theta_{\lambda}\right \vert }{\left \vert
x_{\lambda}z_{\lambda}\right \vert \, \, \left \vert y_{\lambda
}\theta_{\lambda}\right \vert }\longrightarrow \frac{\left \vert y_{\infty
}u\right \vert \, \, \left \vert x_{\infty}\theta_{\infty}\right \vert }{%
\left \vert x_{\infty}u\right \vert \, \, \left \vert
y_{\infty}\theta_{\infty}\right \vert }\mathrm{\ as\ }\lambda \rightarrow
\infty.
\end{equation*}

This completes the proof of the Claim.

We return now to the proof of Theorem \ref{nonisometric}. Assume, on the
contrary, that $F:\left( P,d\right) \rightarrow \left( U,d_{\mathcal{H}%
}\right) $ is an isometry. We write $a^{\prime }$ for the image $F\left(
a\right) $ of a point $a\in P.$ The images of the geodesic rays $\left[
x,x_{\infty }\right] $ and $\left[ y,y_{\infty }\right] $ under $F,$ denoted
by $\left[ x^{\prime },x_{\infty }^{\prime }\right] $ and $\left[ y^{\prime
},y_{\infty }^{\prime }\right] $ respectively, are clearly geodesic rays in $%
U$ determining boundary points $x_{\infty }^{\prime },y_{\infty }^{\prime
}\in \partial U.$ Moreover, 
\begin{equation*}
\lim_{n\rightarrow \infty }  d_{\mathcal{H}}\left( x^{\prime },x_{n}^{\prime }\right) =\infty =\lim_{n\rightarrow \infty } d_{\mathcal{H}}\left( y^{\prime },y_{n}^{\prime }\right) .
\end{equation*}%
Let $\left[ x_{\infty }^{\prime },y_{\infty }^{\prime }\right] $ be the
Euclidean segment joining $x_{\infty }^{\prime }$ and $y_{\infty }^{\prime }.
$

If $\left( x_{\infty}^{\prime},y_{\infty}^{\prime}\right) $ is contained in $%
U$ then $d_{\mathcal{H}}\left( x_{n}^{\prime},y_{n}^{\prime}\right)
\rightarrow \infty$ as $n\rightarrow \infty,$ a contradiction by the Claim.

If $\left( x_{\infty}^{\prime},y_{\infty}^{\prime}\right) \nsubseteq U,$
then by convexity of $U,$ $\left( x_{\infty}^{\prime},y_{\infty}^{\prime
}\right) \subset \partial U,$ that is, $\partial U$ contains at least one
segment. In an identical way, we may perform the same construction starting
with the side $\left[ v,q\right] $ adjacent to $\left[ u,v\right] $ and
geodesic rays $\left[ z,z_{\infty}\right] $ and $\left[ w,w_{\infty }\right] 
$ with $z_{\infty},w_{\infty}\in \left[ v,q\right] .$ It follows that $%
\left( z_{\infty}^{\prime},w_{\infty}^{\prime}\right) $ determines again a
segment in $\partial U.$ If $x_{\infty}^{\prime},y_{\infty}^{\prime
},z_{\infty}^{\prime},w_{\infty}^{\prime}$ were collinear then the geodesic
line $\left( y_{\infty},z_{\infty}\right) \subset P$ would have an image $%
\left( y_{\infty}^{\prime},z_{\infty}^{\prime}\right) \subset U$ connecting
the points $y_{\infty}^{\prime},z_{\infty}^{\prime}$ contained in a segment
in $\partial U,$ which is impossible. It follows that $\partial U$ contains
two distinct segments and, thus, the metric space $\left( U,d_{\mathcal{H}%
}\right) $ is not uniquely geodesic, a contradiction by Proposition \ref%
{basicmetric}. 
\end{proof}
{\bf Acknowledgments} The authors would like to thank the anonymous referee for
very helpful comments and suggestions which significantly improved this paper.\\[3mm]


Conflict of Interest statement: On behalf of all authors, the corresponding
author states that there is no conflict of interest.\newline
\noindent Data Availability Statement: Data sharing not applicable to this
article as no datasets were generated or analyzed during the current study.

\end{document}